\documentclass{llncs}
\usepackage{amsmath}
\usepackage{amssymb}

\usepackage{amsthm}
\usepackage{amsfonts}
\usepackage{enumerate}
\usepackage{llncsdoc}
\usepackage{verbatim}
\usepackage{mathtools} 

\newcommand{\tluste}[1]{\mbox{\mathversion{bold}$ #1 $}}

\newcommand{\mace}[1]{{{#1}}}
\newcommand{\mna}[1]{{\mathcal{#1}}}
 
\newcommand{\omace}[1]{\mbox{$\overline{\mace{#1}}$}} 
\newcommand{\umace}[1]{\mbox{$\underline{\mace{#1}}$}} 
\newcommand{\imace}[1]{\mbox{$\tluste{#1}$}}

\def\Mid#1{#1^c}
\def\Rad#1{#1^\Delta}
\newcommand{\onum}[1]{\mbox{$\overline{{#1}}$}} 
\newcommand{\unum}[1]{\mbox{$\underline{{#1}}$}} 
\newcommand{\ovr}[1]{\mbox{$\overline{{#1}}$}} 
\newcommand{\uvr}[1]{\mbox{$\underline{{#1}}$}}

\newcommand{\ivr}[1]{\mbox{$\tluste{#1}$}} 
 
\newcommand{\inum}[1]{\mbox{$\tluste{#1}$}} 
\newcommand{\R}[0]{{\mathbb{R}}}

\newcommand{\IR}[0]{{\mathbb{IR}}}
\newcommand{\Ss}[0]{\mbox{\large$\Sigma$}}
\newcommand{\Ssunb}[0]{\Ss_{unb}}
\newcommand{\Ssker}[0]{\Ss_{ker}}
\newcommand{\Ssae}[0]{\Ss^{AE}}
\newcommand{\Ssaeunb}[0]{\Ss^{AE}_{unb}}
\newcommand{\Ssaeker}[0]{\Ss^{AE}_{ker}}
\newcommand{\Sstol}[0]{\Ss^{tol}}
\newcommand{\Sstolunb}[0]{\Ss^{tol}_{unb}}
\newcommand{\Sstolker}[0]{\Ss^{tol}_{ker}}

\newcommand{\mmid}[0]{;\,}		

\def\clqq{``}
\def\crqq{''}
\def\quo#1{\clqq{}#1\crqq{}}  
\DeclarePairedDelimiter\parentheses{\lparen}{\rparen}	

\DeclareMathOperator{\inter}{\mathit{int}}	
\DeclareMathOperator{\diag}{diag}

\def\nref#1{$(\ref{#1})$}

\begin{document}

\frontmatter          

\pagestyle{headings}  

\mainmatter              

\title{On unbounded directions of linear interval parametric systems and their extensions to AE~solutions}
\titlerunning{Unboundedness of (AE) linear interval parametric systems}

\author{
  Milan Hlad\'{i}k\inst{}
}
\authorrunning{M.\ Hlad\'{i}k}

\institute{
Charles University, Faculty  of  Mathematics  and  Physics,
Department of Applied Mathematics,
Malostransk\'e n\'am.~25, 118 00,
Prague, Czech Republic,
\email{hladik@kam.mff.cuni.cz} 
}

\maketitle

\begin{abstract}
We consider a system of linear equations, whose coefficients depend linearly on interval parameters. Its solution set is defined as the set of all solutions of all admissible realizations of the parameters. We study unbounded directions of the solution set and its relation with its kernel. The kernel of a matrix characterizes unbounded direction in the real case and in the case of ordinary interval systems. In the general parametric case, however, this is not completely true. There is still a close relation preserved, which we discuss in the paper. Nevertheless, we identify several special sub-classes, for which the characterization remains valid. Next, we extend the results to the so called AE parametric systems, which are defined by forall-exists quantification.
\end{abstract}

\section{Introduction}

Solving systems of linear equations is a fundamental problem of scientific computing. Since real world data are often imprecise due to inexact measurements, estimations, and categorization, we naturally have to reflect such uncertainties in system solving. In this paper, we assume that uncertainty has the form of compact intervals, in which the true values of uncertain quantities are lying. To reflect the structure of practical problems, we allow (linear) dependencies between the coefficients of the system of linear equations. We introduce the problem formally after some necessary notation.

\subsubsection*{Intervals and interval systems.}
An interval matrix is defined as
$$
\imace{A}:=\{A\in\R^{m\times n}\mmid \umace{A}\leq A\leq \omace{A}\},
$$
where $ \umace{A}$ and  $\omace{A}$, $\umace{A}\leq\omace{A}$, are given matrices and the inequality is understood entrywise. 
The midpoint and radius matrices are defined as
$$
\Mid{A}:=\frac{1}{2}(\umace{A}+\omace{A}),\quad
\Rad{A}:=\frac{1}{2}(\omace{A}-\umace{A}).
$$
The set of all interval $m\times n$ matrices is denoted by $\IR^{m\times n}$. Interval vectors are defined analogously.
A system of interval linear equations is a family of systems
$$
Ax=b, \quad A\in\imace{A},\ b\in\ivr{b}.
$$
In this case, the coefficients of the matrix and the right-hand side vector are independent to each other. In many practical problems, however, there are some dependencies between them. That is why we consider a larger class of interval systems below.

\subsubsection*{Linear interval parametric system.}
Let $\ivr{p}=(\inum{p}_1,\dots,\inum{p}_K)^T\in\IR^K$ be a given interval vector, $A^{(1)},\dots,A^{(K)}\in\R^{m\times n}$, and $b^{(1)},\dots,b^{(K)}\in\R^m$.
A linear interval parametric system is a family of systems of linear equations
\begin{align}\label{sysLIPS}
A(p)x=b(p),\ \ p\in\ivr{p},
\end{align}
in which the coefficients depend linearly on parameters, that is, 
\begin{align*}
A(p)=\sum_{k=1}^K p_k A^{(k)},\quad
b(p)=\sum_{k=1}^K p_k b^{(k)}.
\end{align*}
The \emph{(united) solution set} of this parametric system is denoted by $\Ss$ and consists of all solutions of all linear systems corresponding to all possible combinations of parameters $p\in\ivr{p}$, that is,
\begin{align*}
\Ss=\{x\in\R^n\mmid \exists p\in\ivr{p}: A(p)x=b(p)\}.
\end{align*}
A number of works were devoted to computation of a tight enclosure of the solution set; see book \cite{Ska2018} for a survey and very recent papers \cite{HlaSka2019a,Pop2019a,SkaHla2019a}.
However, to find a bounded enclosure, the solution set must be bounded. This motivates us to study unboundedness in this paper.

\subsubsection*{Unbounded directions.}
We investigate the unbounded directions of the solution set~$\Ss$. In particular, we focus on two sets related to unboundedness of~$\Ss$. 
The set of all \emph{unbounded directions} of $\Ss$ is denoted
\begin{align*}
\Ssunb=\left\{y\in\R^n\mmid \exists x\in\Ss\ \forall \alpha\geq0: x+\alpha y\in\Ss\right\}.
\end{align*}
For a real system of linear equations $Ax=b$ which is solvable the set of unbounded directions is simply its kernel $\{y\in\R^n\mmid Ay=0\}$. This motivates us also to investigate the \emph{(united) kernel} of the parametric system \nref{sysLIPS}
\begin{align*}
\Ssker=\{y\in\R^n\mmid \exists p\in\ivr{p}: A(p)y=0\}.
\end{align*}
To the best of our knowledge, a relation between the kernel and unboundedness of the parametric solution set was discussed for the so called tolerable solution set only (see Section~\ref{ssToler}). Namely, Sharaya~\cite{Sha2005a,Sha2005b} addressed the ordinary interval system and Popova~\cite{Pop2015b} investigated the parametric case.
In the context of parametric matrix equations, some conditions for (un)boundedness were proposed by Dehghani-Madiseh \cite{DehDeh2016,DehHla2019a}.

\subsubsection*{Road map.}
In Section~\ref{sGen} we propose general relations between $\Ssunb$ and $\Ssker$. 
In Section~\ref{sSpec} we present two classes of problems, for which both sets coincide.
Section~\ref{sAE} is devoted to a generalization of our results to another types of solution sets of parametric systems, which are characterized by $\forall\exists$-quantification.

\section{General results}\label{sGen}

The parametric solution set $\Ss$ is hard to characterize by an explicit formula. Various special cases were discussed in \cite{AleKre1997,AleKre2003,Hla2007b,Hla2008g,Pop2009,Pop2015} and the general case in \cite{May2017,Pop2012}. We will utilize a result from \cite{Hla2012d}.

\begin{theorem}\label{thmLipsChar}
We have $x\in\Ss$ if and only if
\begin{align*}
w^T\big(A(\Mid{p})x-b(\Mid{p})\big)
\leq \sum_{k=1}^K \big|w^T(A^{(k)}x-b^{(k)})\big|\Rad{p}_k\quad \forall w\in\R^m.
\end{align*}
\end{theorem}

Now, we discuss properties of the sets $\Ssunb$ and $\Ssker$ and relations between them. Obviously, both sets are cones, not necessarily convex.

\begin{theorem}\label{thmUnbKer}
We have $\Ssunb\subseteq\Ssker$.
\end{theorem}

\begin{proof}
Let $y\in\Ssunb$, that is, there is $x^0\in\Ss$ such that
\begin{align*}
w^T\big(A(\Mid{p})(x^0+\alpha y)-b(\Mid{p})\big)
\leq \sum_{k=1}^K \big|w^T(A^{(k)}(x^0+\alpha y)-b^{(k)})\big|\Rad{p}_k
\end{align*}
for every $w\in\R^m$ and every $\alpha\geq0$. Suppose to the contrary that $y\not\in\Ssker$, that is, there is $w\in\R^m$ such that
\begin{align*}
w^T A(\Mid{p})y > \sum_{k=1}^K \big|w^T A^{(k)} y\big|\Rad{p}_k.
\end{align*}
The difference between the left and right-hand sides can be arbitrarily large since we can take an arbitrarily large multiple of~$y$. Thus there is $\alpha^*>0$ such that 
\begin{align*}
w^T A(\Mid{p})\alpha^* y &- \sum_{k=1}^K \big|w^T A^{(k)} \alpha^* y\big|\Rad{p}_k\\
&> w^T\big(b(\Mid{p})-A(\Mid{p})x^0\big)
  + \sum_{k=1}^K \big|w^T(A^{(k)}x^0-b^{(k)})\big|\Rad{p}_k.
\end{align*}
However, this leads to 
\begin{align*}
w^T\big(A(\Mid{p})(x^0+\alpha^* y)-b(\Mid{p})\big)
> \sum_{k=1}^K \big|w^T(A^{(k)}(x^0+\alpha^* y)-b^{(k)})\big|\Rad{p}_k,
\end{align*}
a contradiction.
\end{proof}

\begin{theorem}\label{thmIntUnbKer}
If $\Ss\not=\emptyset$, then $\inter\Ssker\subseteq\inter\Ssunb$.
\end{theorem}

\begin{proof}
Let $y\in\inter\Ssker$, that is, for every $w\in\R^m$ 
\begin{align*}
w^T A(\Mid{p})y < \sum_{k=1}^K \big|w^T A^{(k)} y\big|\Rad{p}_k.
\end{align*}
Let $x^0\in\Ss$. In a similar manner as in the proof of Theorem~\ref{thmUnbKer}, there is a sufficiently large $\alpha^*>0$ such that for every $\alpha\geq\alpha^*$
\begin{align}\label{sysPfThmIntUnbKer}
w^T\big(A(\Mid{p})(x^0+\alpha y)-b(\Mid{p})\big)
< \sum_{k=1}^K \big|w^T(A^{(k)}(x^0+\alpha y)-b^{(k)})\big|\Rad{p}_k.
\end{align}
Notice that the value of $\alpha^*$ does not depend on $w$ since it is sufficient to find $\alpha^*$ such that 
\nref{sysPfThmIntUnbKer} with $\alpha=\alpha^*$ holds 
for each $w$ with $\|w\|=1$ (such $\alpha^*$ exists due to compactness of the unit sphere). 
So $y$ is an unbounded direction of $\Ss$, emerging from the point $x^0+\alpha^*y$. Analogous reasoning as above then shows $y\in\inter\Ssunb$.
\end{proof}

In general we have $\Ssunb\not=\Ssker$, which is illustrated by the following example.

\begin{example}\label{exIneq}
Consider the linear interval parametric system
$$
\begin{pmatrix}1&0\\1&p\end{pmatrix}x=\begin{pmatrix}1\\p\end{pmatrix},\quad
p\in\ivr{p}=[0,1].
$$
Now, unbounded directions of $\Ss$ form a ray, whereas the kernel is a line,
\begin{align*}
\Ssunb &= \{\alpha(1,-1)^T\mmid \alpha\geq0\},\\
\Ssker &= \{\alpha(1,-1)^T\mmid \alpha\in\R\}.
\end{align*}
\end{example}

\section{Special cases}\label{sSpec}

\subsection{Ordinary systems of interval linear equations}

When each interval parameter in \nref{sysLIPS} is associated with one coefficient only, we obtain an ordinary interval system
$$
Ax=b,\quad A\in\imace{A},\ b\in\ivr{b},
$$
where $\imace{A}$ in and interval matrix and $\ivr{b}$ an interval vector. By the Oettli--Prager theorem \cite{OetPra1964}, the solution set $\Ss$ is characterized by the inequality system
\begin{align}\label{ineqOetPra}
|\Mid{A}x-\Mid{b}|\leq\Rad{A}|x|+\Rad{b}.
\end{align}

\begin{proposition}
For ordinary intervals systems, if $\Ss\not=\emptyset$ then $\Ssker=\Ssunb$.
\end{proposition}

\begin{proof}
By \nref{ineqOetPra}, the set $\Ssker$ is characterized by
$$
|\Mid{A}y|\leq\Rad{A}|y|.
$$
Let $s\in\{\pm1\}^n$ and consider the orthant of $\R^n$ associated with the sign vector~$s$, that is, the orthant given by $\diag(s)y\geq0$. For the points lying in this orthant, we can write $|y|=\diag(s)y$. Thus the part of $\Ssker$ lying in this orthant is described by 
$$
|\Mid{A}y|\leq\Rad{A}\diag(s)y,\ \diag(s)y\geq0,
$$
which can be reformulated as
\begin{align}\label{sysPfPropIls}
(\Mid{A}-\Rad{A}\diag(s))y\leq0,\ \ 
-(\Mid{A}+\Rad{A}\diag(s))y\leq0,\ \ 
\diag(s)y\geq0.
\end{align}

Consider now the part of $\Ss$ lying in the orthant $\diag(s)x\geq0$. It is characterized by
\begin{align*}
|\Mid{A}x-\Mid{b}|\leq\Rad{A}\diag(s)x+\Rad{b},\ \diag(s)x\geq0,
\end{align*}
which can be similarly reformulated as
$$
(\Mid{A}-\Rad{A}\diag(s))x\leq\ovr{b},\ \ 
-(\Mid{A}+\Rad{A}\diag(s))x\leq-\uvr{b},\ \ 
\diag(s)x\geq0.
$$
This is an ordinary system of linear inequalities and the set of unbounded directions is simply described by \nref{sysPfPropIls}. This shows $\Ssker=\Ssunb$.
\end{proof}

\subsection{Parametric systems of the first class}

Popova \cite{Pop2009} introduced a notion of the so called parametric systems of the first class. It consists of systems \nref{sysLIPS} such that matrix $(A^{(k)}\mid b^{(k)})$ has at most one nonzero row for each $k=1,\dots,K$. For such systems, $\Ss$ is characterized by the nonlinear system of inequalities
\begin{align*}
|A(\Mid{p})x-b(\Mid{p})|
 \leq\sum_{k=1}^K\Rad{p}_k|A^{(k)}x-b^{(k)}|.
\end{align*}
In Example~\ref{exIneq}, we showed that $\Ssunb\not=\Ssker$ even for parametric systems of the first class. Therefore, in order to achieve equality for a larger class than ordinary interval systems from the previous section, we need to restrict the class a bit more.

Consider the class $\mathcal{C}$ of parametric systems in the form
$$
A(p)x=b(q),
$$
where
\begin{align*}
A(p)=\sum_{k=1}^K p_kA^{(k)},\quad
b(q)=\sum_{\ell=1}^L q_{\ell}b^{(\ell)}q.
\end{align*}
Parameters $p_1,\dots,p_K,q_1,\dots,q_L$ come from their respective interval domains $\inum{p}_1,\dots,\inum{p}_K,\inum{q}_1,\dots,\inum{q}_L\in\IR$. In addition, we assume that  $A^{(k)}$ has at most one nonzero row for each $k=1,\dots,K$ and $b^{(\ell)}$ has at most one nonzero entry for each $k=1,\dots,L$.

\begin{proposition}
For class $\mathcal{C}$, if $\Ss\not=\emptyset$ then $\Ssker=\Ssunb$.
\end{proposition}

\begin{proof}
The set $\Ssker$ is characterized by
\begin{align*}
|A(\Mid{p})y|
 \leq\sum_{k=1}^K\Rad{p}_k|A^{(k)}y|.
\end{align*}
Let $s\in\{\pm1\}^K$ and consider the part of $\R^n$ associated with the sign vector~$s$, that is, the characterized by $s_kA^{(k)}y\geq0$, $k=1,\dots,K$. The part of $\Ssker$ lying in this convex cone is described by 
\begin{align*}
|A(\Mid{p})y|
 \leq\sum_{k=1}^K\Rad{p}_k s_k A^{(k)}y,\ \ s_kA^{(k)}x\geq0,\ k=1,\dots,K
\end{align*}
which equivalently reads
\begin{subequations}\label{sysPfPropPar}
\begin{align}
\sum_{k=1}^K \big(\Mid{p}_k-\Rad{p}_k s_k\big) A^{(k)}y &\leq 0,\ \ 
-\sum_{k=1}^K \big(\Mid{p}_k+\Rad{p}_k s_k\big) A^{(k)}y \leq 0,\\ 
s_kA^{(k)}y&\geq0,\ k=1,\dots,K.
\end{align}
\end{subequations}

Consider now the part of $\Ss$ lying in the convex cone from above. It is characterized by
\begin{align*}
|A(\Mid{p})x-b(\Mid{q})|
 \leq\sum_{k=1}^K\Rad{p}_ks_kA^{(k)}x+\sum_{\ell=1}^L\Rad{q}_{\ell}|b^{(\ell)}|,\ \ 
s_kA^{(k)}x\geq0,\ k=1,\dots,K.
\end{align*}
and it can be reformulated as
\begin{align*}
\sum_{k=1}^K\big(\Mid{p}_k-\Rad{p}_ks_kA^{(k)}\big)x
  &\leq b(\Mid{q}) + \sum_{\ell=1}^L\Rad{q}_{\ell}|b^{(\ell)}|,\\ 
-\sum_{k=1}^K\big(\Mid{p}_k+\Rad{p}_ks_kA^{(k)}\big)x
  &\leq -b(\Mid{q}) + \sum_{\ell=1}^L\Rad{q}_{\ell}|b^{(\ell)}|,\\ 
s_kA^{(k)}x&\geq0,\quad k=1,\dots,K.
\end{align*}
This is an ordinary system of linear inequalities and the set of unbounded directions is described by \nref{sysPfPropPar}. 
\end{proof}

\section{Extensions}\label{sAE}

So far, we considered the united solution set. Now, we extend the definition of the solution of a linear interval parametric system to the so called AE solution; see, e.g., Popova~\cite{Pop2012}. To this end, each interval parameter is associated with a universal or an existential quantifier. Let $\mna{K}^{\forall}$ and $\mna{K}^{\exists}$ denote the index sets of parameters associated with universal and existential quantifiers, respectively. According to the corresponding quantifiers we also split the vector of parameters as $p=(p^{\forall},p^{\exists})$ and their domains as $\ivr{p}=(\ivr{p}^{\forall},\ivr{p}^{\exists})$. Now, the corresponding \emph{AE parametric solution set} is defined as
\begin{align*}
\Ssae=\big\{x\in\R^n\mmid 
 \forall p^{\forall}\in\ivr{p}^{\forall}\, \exists p^{\exists}\in\ivr{p}^{\exists}:
  A(p)x=b(p)\big\}.
\end{align*}
First, we generalize Theorem~\ref{thmLipsChar} to the context of AE solutions.

\begin{theorem}\label{thmAeLipsChar}
We have $x\in\Ssae$ if and only if
\begin{align}\label{ineqAeChar}
w^T\parentheses{A(\Mid{p})x-b(\Mid{p})}
\leq
 \sum_{k\in\mna{K}^{\exists}} \big|w^T(A^{(k)}x-b^{(k)})\big|\Rad{p}_k
 -\sum_{k\in\mna{K}^{\forall}}\big|w^T(A^{(k)}x-b^{(k)})\big|\Rad{p}_k
\end{align}
for every $w\in\R^m$.
\end{theorem}

\begin{proof}
\quo{Only if.} 
Let $x\in\Ssae$ and $w\in\R^m$. Then for every $p^{\forall}\in\ivr{p}^{\forall}$ there is $p^{\exists}\in\ivr{p}^{\exists}$ such that $A(p)x=b(p)$, whence $w^T(A(p)x-b(p))=0$. Thus
\begin{align*}
w^T\parentheses{A(\Mid{p})x-b(\Mid{p})}
&=\sum_{k=1}^K w^T(A^{(k)}x-b^{(k)})(p_k-\Mid{p}_k)\\
&\leq \sum_{k\in\mna{K}^{\exists}} \big|w^T(A^{(k)}x-b^{(k)})\big|\Rad{p}_k \\
&\quad  +\sum_{k\in\mna{K}^{\forall}}w^T(A^{(k)}x-b^{(k)})(p_k-\Mid{p}_k)
\end{align*}
Since the above inequality holds for  every $p\in\ivr{p}^{\forall}$, we can deduce~\nref{ineqAeChar}.

\quo{If.}  
Let $x\not\in\Ssae$. Then there is $p^{\forall}\in\ivr{p}^{\forall}$ such that the system
\begin{align*}
\sum_{k\in\mna{K}^{\exists}}(A^{(k)}x-b^{(k)})p_k=-
\sum_{k\in\mna{K}^{\forall}}(A^{(k)}x-b^{(k)})p_k,
\ \ p_k\in\ivr{p}_k,\ k\in\mna{K}^{\exists},
\end{align*}
has no solution in variables $p_k$, $k\in\mna{K}^{\exists}$. By Farkas lemma~\cite{Schr1998}, the dual system
\begin{subequations}
\begin{align}\label{ineqPfThmAeCharDual1}
u_k-v_k+w^T(A^{(k)}x-b^{(k)})&=0,\quad k\in\mna{K}^{\exists}\\
\label{ineqPfThmAeCharDual2}
\sum_{k\in\mna{K}^{\exists}}(u_k\onum{p}_k-v_k\unum{p}_k)
 -\sum_{k\in\mna{K}^{\forall}}w^T(A^{(k)}x-b^{(k)})p_k&<0,\\
 u_k,v_k&\geq 0,\quad k\in\mna{K}^{\exists}
\end{align}
\end{subequations}
is solvable in variables $u_k,v_k,w$. Without loss of generality we may assume that $u_kv_k=0$ for each $k\in\mna{K}^{\exists}$; otherwise we subtract $\min(u_k,v_k)$ from both $u_k,v_k$ and they still solve the system. We split $\mna{K}^{\exists}=\mna{K}_1\cup\mna{K}_2$ by defining $\mna{K}_1:=\{k\in\mna{K}^{\exists}\mmid u_k>0\}$ and $\mna{K}_2:=\mna{K}^{\exists}\setminus \mna{K}_1$. We subtract from \nref{ineqPfThmAeCharDual2} the $\onum{p}_k$ multiple of \nref{ineqPfThmAeCharDual1} for $k\in\mna{K}_1$, and we subtract the $\unum{p}_k$ multiple of \nref{ineqPfThmAeCharDual1} for $k\in\mna{K}_2$. We arrive at inequality
\begin{align*}
-\sum_{k\in\mna{K}_1}w^T(A^{(k)}x-b^{(k)})\onum{p}_k
&-\sum_{k\in\mna{K}_2}w^T(A^{(k)}x-b^{(k)})\unum{p}_k \\
&-\sum_{k\in\mna{K}^{\forall}}w^T(A^{(k)}x-b^{(k)})p_k<0,
\end{align*}
which we can write as
\begin{align*}
w^T\parentheses{A(\Mid{p})x-b(\Mid{p})}
&>-\sum_{k\in\mna{K}_1}w^T(A^{(k)}x-b^{(k)})\Rad{p}_k
 +\sum_{k\in\mna{K}_2}w^T(A^{(k)}x-b^{(k)})\Rad{p}_k\\
 &\quad -\sum_{k\in\mna{K}^{\forall}}w^T(A^{(k)}x-b^{(k)})(p_k-\Mid{p_k}),
\end{align*}
Since $w^T(A^{(k)}x-b^{(k)})$ is nonpositive for $k\in\mna{K}_1$ and nonnegative for $k\in\mna{K}_2$, we can deduce
\begin{align*}
w^T\parentheses{A(\Mid{p})x-b(\Mid{p})}
&>\sum_{k\in\mna{K}^{\exists}}\big|w^T(A^{(k)}x-b^{(k)})\big|\Rad{p}_k\\
 &\quad -\sum_{k\in\mna{K}^{\forall}}w^T(A^{(k)}x-b^{(k)})(p_k-\Mid{p_k}),\\
&\geq \sum_{k\in\mna{K}^{\exists}}\big|w^T(A^{(k)}x-b^{(k)})\big|\Rad{p}_k
  -\sum_{k\in\mna{K}^{\forall}}\big|w^T(A^{(k)}x-b^{(k)})\big|\Rad{p}_k,
\end{align*}
which contradicts \nref{ineqAeChar}.
\end{proof}

Now, we extend the definitions of the unbounded directions. We introduce the set of all \emph{unbounded directions} of $\Ssae$
\begin{align*}
\Ssaeunb=\left\{y\in\R^n\mmid 
  \exists x\in\Ssae\ \forall \alpha\geq0: x+\alpha y\in\Ssae\right\}
\end{align*}
and the \emph{kernel} of the AE parametric solution set
\begin{align*}
\Ssaeker=\{y\in\R^n\mmid
  \forall p\in\ivr{p}^{\forall}\, \exists p\in\ivr{p}^{\exists}: A(p)y=0\}.
\end{align*}
By using Theorem~\ref{thmAeLipsChar}, it is easy to see that AE quantified analogies of Theorem~\ref{thmUnbKer} and~\ref{thmIntUnbKer} are true.

\begin{theorem}\label{thmUnbKerAe}
We have $\Ssaeunb\subseteq\Ssaeker$.
\end{theorem}

\begin{theorem}
If $\Ssae\not=\emptyset$, then $\inter\Ssaeker\subseteq\inter\Ssaeunb$.
\end{theorem}

\subsection{Parametric tolerable solution set}\label{ssToler}

Popova~\cite{Pop2015b} showed that $\Ssaeunb=\Ssaeker$ provided $\Ssae\not=\emptyset$, the universal quantifiers are in the constraint matrix only and the existential quantifiers are in the right-hand side only. This positioning of quantifiers is an important case of AE solutions and the corresponding solutions are called tolerable solutions. In this section, we slightly extend her result by allowing universal quantifiers to be situated everywhere.

Let $\ivr{p}=(\inum{p}_1,\dots,\inum{p}_K)^T\in\IR^K$, $\ivr{q}=(\inum{q}_1,\dots,\inum{q}_L)^T\in\IR^L$, $A^{(1)},\dots,A^{(K)}\in\R^{m\times n}$, and $b^{(1)},\dots,b^{(K)},d^{(1)},\dots,d^{(L)}\in\R^m$.
Consider a linear interval parametric system 
\begin{align}\label{sysTolerLIPS}
A(p)x=b(p,q),\ \ p\in\ivr{p},\ q\in\ivr{q},
\end{align}
where\begin{align*}
A(p)=\sum_{k=1}^K p_k A^{(k)},\quad
b(p,q)=\sum_{k=1}^K p_k b^{(k)}+\sum_{\ell=1}^L q_{\ell} d^{(\ell)}.
\end{align*}
The \emph{tolerable solution set} of this parametric system is defined as
\begin{align*}
\Sstol=\{x\in\R^n\mmid \forall p\in\ivr{p}\,\exists q\in\ivr{q}: A(p)x=b(p,q)\}.
\end{align*}
The corresponding set of all \emph{unbounded directions} of $\Sstol$ is
\begin{align*}
\Sstolunb=\left\{y\in\R^n\mmid 
  \exists x\in\Sstol\ \forall \alpha\geq0: x+\alpha y\in\Sstol\right\}
\end{align*}
and the \emph{kernel} of the tolerable solution set is
\begin{align*}
\Sstolker=\{y\in\R^n\mmid
  \forall p\in\ivr{p}: A(p)y=0\}.
\end{align*}

\begin{theorem}
If $\Sstol\not=\emptyset$, then $\Sstolunb=\Sstolker$.
\end{theorem}

\begin{proof}
Inclusion $\Sstolunb\subseteq\Sstolker$ follows from Theorem~\ref{thmUnbKerAe}, so we focus on the converse inclusion. Let $y\in\Sstolker$, that is, by Theorem~\ref{thmAeLipsChar} we have
\begin{align*}
w^T A(\Mid{p})y \leq -\sum_{k=1}^K\big|w^TA^{(k)}y\big|\Rad{p}_k
\end{align*}
for every $w\in\R^n$. Since $\Sstol\not=\emptyset$, there is some $x^0\in\Sstol$ such that 
\begin{align*}
w^T\parentheses{A(\Mid{p})x^0-b(\Mid{p})}
\leq
 -\sum_{k=1}^K\big|w^T(A^{(k)}x^0-b^{(k)})\big|\Rad{p}_k
 +\sum_{\ell=1}^L \big|w^Td^{(k)}\big|\Rad{q}_{\ell}
\end{align*}
for every $w\in\R^n$. Putting together, we have that for every $w\in\R^n$ and $\alpha\geq0$
\begin{align*}
w^T&\parentheses*{A(\Mid{p})(x^0+\alpha y)-b(\Mid{p})}\\
&\leq
 -\sum_{k=1}^K\big(\big|w^T(A^{(k)}x^0-b^{(k)})\big|
     + \big|w^TA^{(k)}\alpha y\big|\big)\Rad{p}_k 
  +\sum_{\ell=1}^L \big|w^Td^{(k)}\big|\Rad{q}_{\ell}\\
&\leq
 -\sum_{k=1}^K\big|w^T(A^{(k)}(x^0+\alpha y)-b^{(k)})\big| \Rad{p}_k 
 +\sum_{\ell=1}^L \big|w^Td^{(k)}\big|\Rad{q}_{\ell}.
\end{align*}
By Theorem~\ref{thmAeLipsChar} again, this means that $x^0+\alpha y\in\Sstol$.
\end{proof}

\section{Conclusion}

For systems of real linear equations, the unbounded directions are characterized by the kernel of the constraint matrix. We showed that this equivalence remains valid also for ordinary interval linear equations. For more general linear interval parametric systems, however, this is no longer true. On the other hand, there is still a very close relation between these two concepts for the united solution set and even for a more general AE solution set. Moreover, we identified several special cases for which the equivalence is satisfied: parametric systems of the first class and parametric tolerable solution sets. Identification of other sub-classes may be of interest, too.




\bibliographystyle{abbrv}
\bibliography{coprod_pils_unb}

\end{document}